\documentclass[12pt]{article}
\usepackage[utf8]{inputenc}
\usepackage{graphicx}
\usepackage{caption}
\usepackage{amsmath, amsthm, tikz}
\usepackage{longtable,amssymb}
\usepackage{float}

\usetikzlibrary{graphs,positioning,decorations.pathreplacing}
\usepackage{fullpage}

\newtheorem{theorem}{Theorem}[section]
\newtheorem{lemma}[theorem]{Lemma}

\newtheorem{prop}[theorem]{Proposition}
\newtheorem{conjecture}[theorem]{Conjecture}
\newtheorem{cor}[theorem]{Corollary}

\newtheorem{claim}[theorem]{Claim}

\newcommand*{\myproofname}{Proof of claim}
\newenvironment{claimproof}[1][\myproofname]{\begin{proof}[#1]}{\end{proof}}

\newcommand{\cB}{\mathcal{B}}

\newcommand{\cF}{\mathcal{F}}

\newcommand{\cH}{\mathcal{H}}
\newcommand{\cI}{\mathcal{I}}
\newcommand{\cJ}{\mathcal{J}}

\newcommand{\cT}{\mathcal{T}}

\newcommand{\cX}{\mathcal{X}}
\newcommand{\cV}{\mathcal{V}}
\newcommand{\bc}{\mathbf{c}}

\newcommand{\bi}{\mathbf{i}}

\DeclareMathOperator{\am}{am}

\DeclareMathOperator{\gm}{gm}

\allowdisplaybreaks

\title{Geometric vs Algebraic Nullity for Hyperpaths}
\author{Joshua Cooper
  \and
  Grant Fickes
}

\newcommand{\Addresses}{{
  \bigskip
  \footnotesize

\noindent  Joshua Cooper, \textsc{Department of Mathematics, University of South Carolina,
    Columbia, SC 29208 USA}\par\nopagebreak
  \textit{E-mail address}, Joshua Cooper: \texttt{cooper@math.sc.edu}

  \medskip

\noindent  Grant Fickes, \textsc{Department of Mathematics, University of South Carolina,
    Columbia, SC 29208 USA}\par\nopagebreak
  \textit{E-mail address}, Grant Fickes: \texttt{gfickes@email.sc.edu}

}}
\date{\today}

\begin{document}

\maketitle

\begin{abstract}
    We consider the question of how the eigenvarieties of a hypergraph relate to the algebraic multiplicities of their corresponding eigenvalues.  Specifically, we (1) fully describe the irreducible components of the zero-eigenvariety of a loose $3$-hyperpath (its ``nullvariety''), (2) use recent results of Bao-Fan-Wang-Zhu to compute the corresponding algebraic multiplicity of zero (its ``nullity''), and then (3) for this special class of hypergraphs, verify a conjecture of Hu-Ye about the relationship between the geometric (multi-)dimension of the nullvariety and the nullity. 
\end{abstract}

\section{Introduction}

We begin with two questions:
\begin{enumerate}
    \item What is the combinatorial meaning of the multiplicity of the zero eigenvalue of a (hyper)graph?
    \item What is the relationship between the various notions of ``multiplicity'' for an eigenvalue?
\end{enumerate}
One may combine these two questions by asking, ``What is the combinatorial meaning of {\it each notion} of the multiplicity of the zero eigenvalue of (hyper)graphs?''  For the Laplacian matrix $L(G) = D(G) - A(G)$ of a graph, in the 1970s, Fiedler showed that the multiplicity -- in both the algebraic and geometric senses -- of the zero eigenvalue is equal to the number of components of $G$.  Thus it is natural to ask this same question about the seemingly simpler adjacency matrix $A(G)$, and indeed considerable attention has been given to Question 1 (e.g., \cite{Cvetkovic72,Fiorini05,Gutman11,Sciriha07,Wang20}).  Because $A(G)$ is real symmetric and therefore diagonalizable, the answer to Question 2 is simple for a graph, however: they agree.

In contrast, these questions are nearly untouched for hypergraphs.  The first question has been investigated for some special graphs -- for example, \cite{Bao20} implicitly provides an algorithm for computing the algebraic multiplicity of zero as an eigenvalue of a hyperpath.  In a related vein, \cite{Clark18} analyzes which eigenvectors corresponding to the zero eigenvalue of a subgraph of $G$ are also such ``null eigenvectors'' for $G$.  The second question is also almost entirely unexplored for hypergraphs, and Sturmfels observed (see \cite{Hu16}) that the relatively straightforward linear eigenspaces of matrices become complicated ``eigenvarieties'' when one passes to adjacency tensors/hypermatrices to study hypergraphs.  Hu and Ye \cite{Hu16} take up this matter in earnest and pose a conjecture about the relationship between the (multi-)dimension of such varieties and their multiplicities as roots of a hypermatrix's characteristic polynomial; these are natural choices for analogizing ``geometric'' and ``algebraic'' multiplicity, respectively, and the conjecture is an attempt to generalize the fact that the geometric multiplicity of a matrix eigenvalue is bounded above by its algebraic multiplicity.  Another notable contribution \cite{Fan19} by Fan-Bao-Huang investigated properties of the eigenvariety associated with the spectral radius of a hypergraph (and, more generally, certain hypermatrices/tensors).

The aforementioned Hu-Ye Conjecture can be stated as follows; definitions follow below.  Let $\am(\lambda)$ be the algebraic multiplicity of $\lambda$ as an eigenvalue of the hypermatrix $M$.  Let $V_\lambda^1,\ldots,V_\lambda^\kappa$ denote the irreducible components of $V_\lambda$, the eigenvariety correponding to $\lambda$.

\begin{conjecture}[\cite{Hu16}] \label{conj:HuYe} For any order-$k$ hypermatrix $M$, define
$$
\gm(\lambda) := \sum_{j=1}^\kappa \dim(V_\lambda^j) (k-1)^{\dim(V_\lambda^j)-1} 
$$
Then $\gm(\lambda) \leq \am(\lambda)$.
\end{conjecture}

Here we verify this for the zero eigenvalue of a simple class of $3$-uniform hypergraphs -- sometimes called ``loose paths'' or ``linear hyperpaths'' -- by obtaining an explicit description of the irreducible components of their nullvarieties, using this to obtain a generating function that encodes said irreducible components' dimensions, using results from \cite{Bao20} to obtain an explicit expression for the multiplicity of zero as a root of their characteristic polynomials, and comparing the resulting quantities to confirm the conjecture in this special case.  \\

We briefly define the multilinear algebra and spectral hypergraph theory terminology and notation used throughout the paper.  More detailed information and references can be found in \cite{Cooper12,Fan19}.  An order-$k$ {\em hypermatrix}\footnote{Variously known as a ``tensor'' in some literature.} $M$ over a ring $R$ is a $k$-dimensional array of values $M_{i_1\cdots i_k} \in R$ (usually $R = \mathbb{C}$), which we often identify with the function $M : (i_1,\ldots,i_k) \mapsto M_{i_1\cdots i_k}$. A hypermatrix is {\em cubical} if the $i_j$, $j = 1,\ldots,k$, all belong to the same index set $\cI$, in which case we say that its {\em dimension} is $|\cI|$, and a cubical hypermatrix is {\em symmetric} if, for every permutation $\sigma$ of $\cI$ and $\bi = (i_1,\ldots,i_k) \in \cI^k$, $M_\bi = M_{\sigma(\bi)}$, where $\sigma(\bi) = (\sigma(i_1),\ldots,\sigma(i_k))$.  An order-$k$ cubical hypermatrix $M$ of dimension $n$ over $R$ gives rise to a homogeneous $k$-form $Mx^k$, where $x = (x_1,\ldots,x_n)$, given by $\sum_{\bi \in [n]^k} M_\bi x^\bi$, where $x^\bi$ denotes $\prod_{j=1}^k x_{i_j}$ if $\bi = (i_1,\ldots,i_k)$.   The {\em symmetric hyperdeterminant} $\det(M)$ of a symmetric hypermatrix $M$ over $R = \mathbb{C}[\{x_\bi\}_{\bi \in [n]^k}]$ is the unique monic irreducible polynomial over $R$ which vanishes if and only if $\nabla (Mx^k) = \mathbf{0}$ for some nonzero vector $x \in \mathbb{C}^n$.  The {\em identity} hypermatrix $I$ of rank $k$ and order $n$ is the function so that $I(i_1,\ldots,i_k) = 1$ if $i_1 = \cdots = i_k$ and $0$ otherwise.  Write $\lambda M$ for the hypermatrix whose $\bi$ entry is $\lambda M_{\bi}$ for each valid multi-index $\bi$.  Then the {\em characteristic polynomial} of $M$ is $\phi_M(\lambda) := \det(\lambda I - M) \in \mathbb{C}[\lambda]$.  The (homogeneous) {\em spectrum} of $M$ is the multiset of roots of $\phi_M(\lambda)$; the elements $\lambda$ of the adjacency spectrum of $M$ are referred to as {\em eigenvalues} of $M$, and any nonzero $x$ so that $\nabla[(M-\lambda I)x^k] = 0$ is a {\em corresponding eigenvector}.  The set of all eigenvectors corresponding to an eigenvalue $\lambda$ of a hypermatrix $M$ of dimension $n$ is its $\lambda$-eigenvariety $\cV_\lambda$.  Then $\cV_\lambda$ is an affine algebraic variety in $\mathbb{C}^n$; indeed, since the equations defining eigenvectors are homogeneous, $\cV_\lambda$ can also be viewed as a projective variety, although we adhere to the affine perspective presently.  The multiplicity of $\lambda$ as a root of $\phi_M(\lambda)$ is its {\em algebraic muliplicity}, and the dimension of the variety $\cV_\lambda$ is its {\em geometric multiplicity}.  {\bf Since the $0$-eigenvariety of a matrix $M$ -- i.e., a hypermatrix of order $k=2$ -- is its nullspace, we refer to the $0$-eigenvariety as the {\em nullvariety} of $M$.  We also refer to the algebraic multiplicity of $0$ as the {\em nullity} of $M$.}  

A (uniform) {\em hypergraph} $\cH$ of rank $k$ is a pair $(V,E)$, where $E \subset \binom{V}{k}$. The {\em adjacency hypermatrix} of a hypergraph $\cH$ is the symmetric hypermatrix $A(\cH) : V^k \rightarrow \mathbb{C}$ so that $A(\cH)_{v_1 \cdots v_k}$ is $1/(k-1)!$ if $\{v_1,\ldots,v_k\} \in E(\cH)$ and $0$ otherwise. 
The $k$-form $p(x_1,\ldots,x_k)=A(\cH)x^k$ is sometimes known as the {\em Lagrangian polynomial} of $\cH$; the coordinate $\partial p/\partial x_i$ of $\nabla A(\cH)x^k$ is ($k$ times) the Langrangian polynomial of the {\em link} of vertex $v_i$ in $\cH$, i.e., the hypergraph whose edges are $\{e \setminus \{x_i\} | v_i \in e \in E(\cH)\}$.  
We will often abuse notation slightly and refer to the multilinear algebraic properties of $A(\cH)$ by describing them as properties of $\cH$ instead.  For example, the (adjacency) {\em spectrum} of a hypergraph $\cH$ is the spectrum of $A(\cH)$, the {\em nullvariety} of $\cH$ is the nullvariety of $A(\cH)$, and $\phi_\cH(\lambda) := \phi_{A(\cH)}(\lambda)$.  A loose hyperpath $P_n^k$ is the $k$-uniform hypergraph on $n$ edges $\{e_1,\ldots,e_n\}$ so that, for $i \neq j$, $|e_i \cap e_j|$ is $1$ if $|i-j|=1$ and $0$ otherwise.  We label the vertex set $V(P_n^k)$ with $\{v_1,\ldots, v_{(k-1)n+1}\}$ so that $e_j = \{v_{(k-1)(j-1)+1},\ldots,v_{(k-1)j+1}\}$ for $j \in [n]$.

Throughout, we also write $\cV(S)$ for the affine variety over $\mathbb{C}$ defined as the zero locus of the set of polynomials $S$, and $\cV(p)$ for $\cV(\{p\})$.  We write $\langle S \rangle$ for the ideal generated by a set of polynomials $S$, and call $S$ {\em irredundant} if $\langle S' \rangle \neq \langle S \rangle$ when $S' \subsetneq S$ is any proper subset. Also, given $p \in \mathbb{C}[x_1,\ldots,x_m]$ and a vector $\mathbf{c} \in \mathbb{C}^m$, we will sometimes say ``$\bc$ satisfies $p$'' if $p(\bc) = 0$.\\

In the next section, we enumerate the irreducible components of the nullvariety of $P_n^3$ and capture their count and the quantity $\gm(0)$ as a generating function.  The following section repeats this exercise, but for the nullity $\am(0)$ of $P_n^3$ -- in fact, more generally $P_n^k$ for $k \geq 3$.  The last section compares these two functions of $n$, verifying the Hu-Ye Conjecture for the zero eigenvalue of $P_n^3$.

\section{Null Variety for Rank-$3$ Loose Hyperpaths}\label{Sec:GeoMult}

We examine the ``geometric multiplicity'' of the zero eigenvalue for a hypergraph $\mathcal{H}$, or more accurately, the multiset of dimensions of irreducible components of the corresponding nullvariety. 
Our strategy will be as follows.  First, we describe the ideal whose zero locus is the nullvariety, generated by the Lagrangian polynomials of the links of all vertices.  Each of the degree-one vertices contributes a polynomial to the ideal which is a simple product of variables.  Thus, the vanishing of these monomials reduces to the vanishing of each of their constituent variables, one at a time.  Considering the set of possible vanishing monomials -- which correspond to vertices/coordinates where portions of the nullvariety are zero -- results in substantial simplification of the set of polynomials in the ideal.  Thus, taking the union of all such vanishing set possibilities gives a decomposition of the nullvariety into simpler subvarieties.  We then analyze these subvarieties to show that they are irreducible.  Next, it is necessary to identify which such irreducible subvarieties are maximal in order to obtain an irredundant list of irreducible components.  Finally, we describe the multiset of these components' dimensions by counting the number of polynomials determining them, leading to an expression for $\gm(0)$.
As a warm-up, and for completeness, we start with the one-edge and two-edge hyperpaths.

\subsection{Small Cases}
\begin{prop}\label{prop:oneedge}
The $3$-uniform hyperedge $\cH = P_1^3$ has three irreducible components of dimension $1$, and $\gm(0) = 3$. 
\end{prop}
\begin{proof}
Let the vertices of $\cH$ be $v_1, v_2, v_3$. Given a null vector $x$, if the adjacency tensor of $\cH$ is $A = A(\cH)$, then 
the $i$-th component of $Ax^{2}$ is given by $\sum_{\{i,j,k\}\in E(\cH)} x_jx_k = x_1 x_2 x_3/x_i$.
Since $x$ is a null vector, we have $x_1x_2 = x_1x_3 = x_2x_3 = 0$, and we consider the variety $V_0 \subset \mathbb{C}^3$ in three-dimensional affine space defined by these equations. If $p,q$ are polynomials, then $\cV(p,q) = \cV(p)\cap \cV(q)$ and $\cV(pq) = \cV(p)\cup \cV(q)$. Therefore, we have the following. 
\begin{align*}
\cV(x_1x_2,x_1x_3,x_2x_3) &= \cV(x_1x_2) \cap \cV(x_1x_3) \cap \cV(x_2x_3) \\
&= \left[\cV(x_1)\cup\cV(x_2)\right] \cap \left[\cV(x_1)\cup\cV(x_3)\right] \cap \left[\cV(x_2)\cup\cV(x_3)\right]
\end{align*}
This is equal to the union over all choices of $\cV(x_i)\cap\cV(x_j)\cap\cV(x_k) = \cV(x_i,x_j,x_k)$ where $i\in\{1,2\}$, $j\in\{1,3\}$, and $k\in\{2,3\}$. Thus, maximal subvarieties of $V_0$ correspond to minimal sets $\{i,j,k\}$ given these conditions, i.e.,
$$
\left[\cV(x_1)\cup\cV(x_2)\right] \cap \left[\cV(x_1)\cup\cV(x_3)\right] \cap \left[\cV(x_2)\cup\cV(x_3)\right] = \cV(x_1,x_2) \cup \cV(x_1,x_3) \cup \cV(x_2,x_3).
$$
Since $\cV(x_i,x_j)$ is the $x_k$-axis, $V_0$ is the union of three lines.
\end{proof}

\begin{prop}\label{prop:twoedge}
If $\cH = P_2^3$, then $V_0$ has one component of dimension $1$ and another of dimension $3$, so that $\gm(0) = 13$. 
\end{prop}
\begin{proof}
Let the vertices of $\cH$ be $v_1, v_2, v_3, v_4, v_5$. Let $x$ be a null vector. The equations defining $V_0 \subset \mathbb{C}^5$ are $x_1x_3 = x_2x_3 = x_1x_2 + x_4x_5 = x_3x_4 = x_3x_5$. Decompose this system as follows: 
$$
V_0 = \cV(x_1x_3,x_2x_3,x_3x_4,x_3x_5) \cap \cV(x_1x_2+x_4x_5).
$$
In the first conjunct, we have intersections of unions, namely 
$$
\cV(x_1x_3,x_2x_3,x_3x_4,x_3x_5) = \left[ \cV(x_1) \cup \cV(x_3) \right] \cap \left[ \cV(x_2) \cup \cV(x_3) \right] \cap \left[ \cV(x_3) \cup \cV(x_4) \right] \cap \left[ \cV(x_3) \cup \cV(x_5) \right].
$$
Expand the expression on the right to obtain the union over all choices of $\cV(x_i)\cap\cV(x_j)\cap\cV(x_k)\cap\cV(x_\ell) = \cV(x_i,x_j,x_k,x_\ell)$ where $i\in\{1,3\}$, $j\in\{2,3\}$, $k\in\{3,4\}$, and $\ell\in\{3,5\}$. The union $\bigcup_{\{i,j,k,\ell\}} \cV(x_i,x_j,x_k,x_\ell)$ is the union over the minimal sets $\{i,j,k,\ell\}$ of this form, i.e., 
\begin{align*}
\left[ \cV(x_1) \cup \cV(x_3) \right] \cap \left[ \cV(x_2) \cup \cV(x_3) \right] & \cap \left[ \cV(x_3) \cup \cV(x_4) \right] \cap \left[ \cV(x_3) \cup \cV(x_5) \right] \\
=& \cV(x_3) \cup \cV(x_1,x_2,x_4,x_5). 
\end{align*}
The second variety has dimension four, while the first variety is the $x_3$ axis. It remains to intersect each such set with $\cV(x_1x_2 + x_3x_4)$. Note that $\cV(x_1,x_2,x_4,x_5)\subseteq \cV(x_1x_2+x_4x_5)$, so that intersection yields $\cV(x_1,x_2,x_4,x_5)$. The intersection of $\cV(x_1x_2+x_4x_5)$ and $\cV(x_3)$ gives $\cV(x_1x_2+x_4x_5,x_3)$, which is a three-dimensional variety. Thus,
$$
V_0 = \cV(x_1,x_2,x_4,x_5) \cup \cV(x_1x_2+x_4x_5,x_3),
$$
which is the union of a one-dimensional and a three-dimensional irreducible component.
\end{proof}

\subsection{General $3$-uniform case}

We now generalize the above approach to all $3$-uniform loose hyperpaths.  Define $p_k$ to be $x_{k-2}x_{k-1}+x_{k+1}x_{k+2}$ for some integer $k$. For integer $n\geq 1$, define $A'_n := \{2k+1: 1\leq k\leq n-1\}$, and let $A_n = A_n'\setminus\{3,2n-1\}$. Define $\cF_n$ be the collection of ``Fibonacci subsets'' of $A'_n$, i.e., sets containing at least one of each two consecutive elements:
$$
\cF_n = \{ S \subset A'_n : \forall k \in [n-2], (2k+1 \in S) \vee (2k+3 \in S)\}
$$
Let $S$ be any element of $\cF_n$.  We say that a set of polynomials $B \subset \{x_i : i \in [2n+1]\} \cup \{p_i : i \in A'_n\}$ is $S$-admissible if it can be obtained in the following manner.  Define $U_i$, $i = 1,2,3,4$, in the following way. 
\begin{enumerate}
\item $U_1 = \{x_i:i\in S\}$
\item $U_2 = \begin{cases}
\{x_1,x_2\} & \mbox{if }x_3\notin U_1, x_5 \in U_1 \\
\{x_1\}\mbox{ or }\{x_2\} & \mbox{if } \{x_3,x_5\}\subseteq U_1 \\
\{p_{3}\} & \mbox{if } x_{3} \in U_1, x_{5}\notin U_1
\end{cases} $
\item $U_3 = \begin{cases}
\{x_{2n},x_{2n+1}\} & \mbox{if }x_{2n-1}\notin U_1, x_{2n-3}\in U_1 \\
\{x_{2n}\}\mbox{ or }\{x_{2n+1}\} & \mbox{if } \{x_{2n-3},x_{2n-1}\}\subseteq U_1 \\
\{p_{2n-1}\} & \mbox{if } x_{2n-1} \in U_1, x_{2n-3}\notin U_1
\end{cases}$
\item $U_4 = \bigcup_{a\in A} \begin{cases}
\emptyset & \mbox{if } \{x_{a-2},x_{a+2}\} \subseteq U_1 \\
\{x_{a+1}\} & \mbox{if } x_{a-2} \in U_1 ,x_{a+2} \notin U_1 \\
\{x_{a-1}\} & \mbox{if } x_{a-2} \notin U_1 ,x_{a+2} \in U_1 \\
\{p_a\} & \mbox{if } \{x_{a-2}, x_{a+2}\} \subseteq \{x_j: j\in A_n'\}\setminus U_1 \\
\end{cases}$
\end{enumerate}
Note that the only choices that do not depend only on $S$ arise from cases of $U_2$ and $U_3$. If we let $\mathcal{U}_i$ denote the collection of all allowable $U_i$, $i=2,3$, then $\cT_S = \{U_1\cup U_2 \cup U_3 \cup U_4 : U_2 \in \mathcal{U}_2, U_3 \in \mathcal{U}_3 \}$ is the collection of $S$-admissible sets. We also remark that for each $B\in\cT_S$, $B\subseteq \mathbb{C}[x_1,\dotsc,x_{2n+1}]$. Define $I_B$ as the ideal in $\mathbb{C}[x_1,\dotsc,x_{2n+1}]$ generated by the polynomials in $B$. Furthermore, let $\mathcal{I}_n$ denote the collection of all such ideals generated by $S$-admissible sets in $\mathcal{F}_n$, i.e., 
$$
\mathcal{I}_n = \{I_B : B \in \cT_S \text{ for some } S\in\mathcal{F}_n \}. 
$$
Before proceeding, we note the following useful fact.

\begin{prop}[Prop. 5.20 in \cite{Milne17}] \label{prop:cartesian} If $V$ and $W$ are irreducible affine varieties over an algebraically closed field, then $V \times W$ is as well.
\end{prop}

In fact, the way we will often use Proposition \ref{prop:cartesian} is: if $I \subset \mathbb{C}[x_1,\ldots,x_n]$ and $J \subset \mathbb{C}[y_1,\ldots,y_m]$ are prime ideals and $I',J'$ are the ideals they generate in $\mathbb{C}[x_1,\ldots,x_n,y_1,\ldots,y_m]$, respectively, then $I'+J'$ is also a prime ideal, and $\cV(I'+J') = \cV(I) \times \cV(J)$.  The following lemma establishes that the ideals in $\mathcal{I}_n$ are prime.

\begin{lemma}\label{Lem:IisPrime}
For $n \geq 3$ and each $I_B\in \mathcal{I}_n$, $I_B$ is a prime ideal in $\mathbb{C}[x_1,\dotsc,x_{2n+1}]$, and $B$ is an irredundant set of generators for it.
\end{lemma}
\begin{proof}  First, since polynomial rings over $\mathbb{C}$ are UFDs, primality is equivalent to irreducibility throughout.  Note that the generators of $I_B$ are a finite collection of variables and polynomials of the form $p_k$ for some odd integer(s) $k$. Let $\mathcal{X} = \{x_i: x_i\in I_B\}$ and $\mathcal{X}' = \{x_1,\dotsc,x_{2n+1}\}\setminus \mathcal{X}$. Furthermore, let $\mathcal{K} = \{p_k:p_k\in I_B\}$. By Proposition \ref{prop:cartesian}, it suffices to show the primality of the ideal generated by $\mathcal{K}$ in the ring $\mathbb{C}[\mathcal{X}']$, since the variables appearing in $\mathcal{K}$ are disjoint from those of $\mathcal{X}$. 
The base case $|\mathcal{K}| = 1$ holds if and only if the polynomial in $\mathcal{K}$ is irreducible. Let $i\in \mathbb{Z}$ so that $p_i\in\mathcal{K}$. It is easy to see that $p_i = x_{i-2}x_{i-1} + x_{i+1}x_{i+2}$ is irreducible. Fix an integer $k \geq 1$ and suppose that the result holds for all $\mathcal{K}'$ with $|\mathcal{K}'| = k$. Let $|\mathcal{K}| = k+1$ and let $p_i$ be any element of $\mathcal{K}$. By the induction hypothesis, $\mathcal{K}\setminus\{p_i\}$ generates a prime ideal. From here we split into the following two cases. \\

Case $1$: The variables of $p_i$ are disjoint from those of $\mathcal{K}\setminus\{p_i\}$. As noted above, $p_i$ generates a prime ideal in $\mathbb{C}[x_{i-2},x_{i-1},x_{i+1},x_{i+2}]$, so it also generates a prime ideal in $\mathbb{C}[\mathcal{X}]$. Moreover, the induction hypothesis gives that $\mathcal{K}\setminus\{p_i\}$ generates a prime ideal in $\mathbb{C}[\mathcal{X}'\setminus\{x_{i-2},x_{i-1},x_{i+1},x_{i+2}\}]$, further implying that $\mathcal{K}\setminus\{p_i\}$ generates a prime ideal in $\mathbb{C}[\mathcal{X}]$ by Proposition \ref{prop:cartesian}.\\

Case $2$: Some variables of $p_i$ also occur as variables of polynomials in $\mathcal{K}\setminus\{p_i\}$. Since $i$ is odd, $i-1$ and $i+1$ are even. Moreover, the variables $x_{i-1}$ (and $x_{i+1}$) appear in no other polynomial of $\mathcal{K}$, since $p_i,p_{i-2}\in\mathcal{K}$ (respectively, $p_i$ and $p_{i+2}$) implies both $i$ and $i-2$ (respectively, $i$ and $i+2$) are outside the set $B$ used to generate the original ideal $I_B$, contradicting that $B$ is generated by a Fibonacci subset of $A_n'$. Therefore, the only overlap in variables comes from $x_{i-2}$ and $x_{i+2}$. 

Let $X$ be the collection of variables in $p_i$ that also appear in polynomials of $\mathcal{K}\setminus\{p_i\}$. Define $Y := \{x_{i-2},x_{i-1},x_{i+1},x_{i+2}\}\setminus X$, and let $Z$ be the collection of variables in polynomials of $\mathcal{K}$ except the variables contained in $X$. Define a collection of new variables $X' := \{x_m' : x_m \in X\}$. Let the polynomial $p_i'$ be $p_i$ evaluated at the variables of $X'$ and $Y$, where each input variable matches the index of the existing variable. Let $I$ be the ideal generated by $\mathcal{K}\setminus\{p_i\}$. The induction hypothesis gives that $I$ is prime. The ideal $\langle p_i'\rangle$ is prime because $p_i'$ is irreducible. Proposition \ref{prop:cartesian} gives the primality of the ideal generated by $I+\langle p_i'\rangle$. Let $\sigma:\mathbb{C}[X\cup Z] \times \mathbb{C}[X'\cup Y] \to \mathbb{C}[X\cup Y\cup Z]$ be the quotient homomorphism $\sigma:f \mapsto f+ \langle \{x_i - x_i' : x_i \in X\} \rangle$. Clearly, $\sigma$ is surjective, so Proposition $3.34$b in \cite{Milne17} (that surjective homomorphisms preserve primality) completes the proof of primality. 

The second claim, that $B$ is irredundant, is straightforward to check from the conditions defining $U_i$, $i = 1,\ldots,4$: with respect to the variable ordering
$$
x_1 \prec x_3 \prec \cdots \prec x_{2n+1} \prec x_2 \prec x_4 \prec \cdots \prec x_{2n},
$$
the set $B$ is triangular (its $\prec$-main variables are distinct), so form a basis of $I_B = \langle B \rangle$. 
\end{proof}

If we let $\mathcal{W}_n$ denote the collection of affine varieties generated by ideals of $\mathcal{I}_n$, i.e., $\mathcal{W}_n := \bigcup_{I\in\mathcal{I}_n} \mathcal{V}(I)$, then the previous lemma implies that all varieties in $\mathcal{W}_n$ are irreducible.  However, some of these varieties may not be inclusion-maximal, so they are not irreducible components, a matter we address presently. 

\begin{lemma}\label{Lem:MaxComps}
Let $I_V$ denote the ideal in $\mathcal{I}_n$ which generates the variety $V$. Furthermore, let $\Theta_V$ denote the collection of all maximal sets of consecutive odd-indexed $x_i \in I_V$ whose indices are contained in $A_n'$. Then, the variety $V\in\mathcal{W}_n$ is inclusion-maximal if and only if $\Theta_V$ does not contain a set with odd cardinality $m\geq 3$. 
\end{lemma}
\begin{proof}
Suppose first that there exists an odd $m\geq 3$ so that $X:=\{x_a,x_{a+2},\ldots,x_{a+2(m-1)}\}\in\Theta_V$. Let $B$ be the generating set for $I_V$ which corresponds to an $S$-admissible set for some $S\in\mathcal{F}_n$. If $a>3$, then the maximality of $X$ implies $x_{a-2}\notin I_V$, giving that $x_{a-1}\in I_V$ by condition (4) in the definition of $S$-admissible. On the other hand, if $a = 3$, then condition (2) gives the presence of either $x_1$ or $x_2$ in $I_V$. In either case, there exists $q_1\in\mathbb{N}$ so that $x_{q_1}\in I_V \cap \{x_{a-2},x_{a-1}\}$. Similarly, there exists $q_2$ so that $x_{q_2}\in I_V\cap \{x_{a+2m-1},x_{a+2m}\}$. Now, define $X' := \{x_{q_1},x_{q_2}\}\cup \{x_{a+2},x_{a+6},\dotsc,x_{a+2(m-2)}\}$, which is well-defined since $|X|$ is odd, and, let $P' = \{p_a,p_{a+4},\dotsc,p_{a+2(m-1)}\}$. Note that $|X'| = |P'| + 1$, so that $|B| > |(B\setminus X')\cup P'|$. Inspection shows that $(B\setminus X')\cup P'$ is an $S$-admissible set for some $S\in\mathcal{F}_n$. Moreover, if every polynomial in $B$ yields $0$ when evaluated at a tuple $(c_1,c_2,\dotsc,c_{2n+1}) \in \mathbb{C}^{2n+1}$, then $(c_1,c_2,\dotsc,c_{2n+1})$ is also a common zero of all polynomials in $(B\setminus X')\cup P'$, since all polynomials of $P'$ evaluate to zero if those of $X\cup\{x_{q_1},x_{q_2}\} \subseteq B$ do as well. Then $V$ is not maximal. 

It remains to establish the converse. If $n \leq 3$, it is straightforward to check that the varieties in $\mathcal{W}_n$ are maximal. Suppose now that $n\geq 4$ and that $V\in\mathcal{W}_n$ is not maximal, so there exists $V'\in\mathcal{W}_n$ with $V \subsetneq V'$. Let $B$ and $B'$ be the admissible sets which generate $I_V$ and $I_{V'}$ respectively, meaning $B$ and $B'$ also generate $V$ and $V'$. Since $V \subset V'$, if values for $x_1,\dotsc,x_{2n+1}$ are chosen so that all polynomials in $B$ are zero, then all the polynomials in $B'$ are also zero for the same choice of values for $x_1,\dotsc,x_{2n+1}$. By the definition of admissible sets, $B$ and $B'$ are each minimal generating sets of their respective ideals, and additionally $B\cap B' \notin\{B,B'\}$, i.e., neither is a subset of the other. 

Next, we establish the following claim regarding the inclusion of single-variable monomials between $B$ and $B'$. Let $i\in [2n-1]$. 
\begin{claim}
\item If $x_i\notin B$, then $x_i\notin B'$. 
\end{claim}
\begin{claimproof}
Suppose $i\in [2n-1]$ and $x_i\notin B$. Let $\mathbf{c} = (c_1,\ldots,c_{2n+1}) \in V$.  If $c_i \neq 0$, then $x_i\notin B'$, as otherwise $\mathbf{c} \not \in V'$, contradicting that $V \subset V'$. Suppose now that $c_i = 0$. The following cases construct another point $\mathbf{c}'$ so that $\mathbf{c}'\in V$ with $c'_i\neq 0$, again obtaining a contradiction to $V \subset V'$.

Case $1$: $i\in \{1,2,2n,2n+1\}$. Without loss of generality, suppose $i = 1$, and note that the only polynomials of any admissible set in which $x_{1}$ occurs are $x_{1}$ and $p_{3}$, and, in this case, $x_1 \not \in B$. If $p_3\in B$, define $\mathbf{c}'$ so that $c'_i = c_i$ for $i \not \in \{1,2\}$, but $c'_1 = 1$ and $c'_{2} = -x_{4} x_{3}$.   The choice of $c'_2$ gives $p_3(\mathbf{c'}) = 0$. Since $p_3\in B$ implies $x_2\notin B$, $\mathbf{c}' \in V$. If $p_3\notin B$, define $\mathbf{c}'$ so that $c'_j = c_j$ for $j \neq 1$, but $c'_1 = 1$. All polynomials of $B$ are  satisfied by $\mathbf{c}'$. \\

Case $2$: $3 \leq i\leq 2n-1$ and $i$ odd. Note that the only possible polynomials containing $x_i$ are $p_{i-2}$, $p_{i+2}$, and $x_i$. By assumption, $x_i\notin B$, leaving only $p_{i-2}$ and $p_{i+2}$. Define $\mathbf{c}'$ so that $c'_j = c_j$ for $j \not \in \{i-1,i,i+1\}$ and $c_i' = 1$. If $p_{i-2}\in B$ (resp.~$p_{i+2}\in B$), define $c_{i-1}' = -x_{i-4} x_{i-3}$ (resp.~$c_{i+1}' = -x_{i+4} x_{i+3}$), so that $p_{i-2}(\mathbf{c}')=0$ (resp.~$p_{i+2}(\mathbf{c}')=0$). The existence of $p_{i-2}\in B$ (resp.~$p_{i+2}$) implies $x_{i-1}\notin B$ (resp.~$x_{i+1}\notin B$). Clearly, $p_i$ is the only other polynomial containing either $x_{i-1}$ or $x_{i+1}$, but $x_i\notin B$ implies $x_{i-2},x_{i+2}\in B$, further giving that $p_i\notin B$. Therefore, all polynomials of $B$ are satisfied by $\mathbf{c}'$.\\

Case $3$: $3 \leq i\leq 2n-1$ and $i$ even. Note that the only possible polynomials containing $x_i$ are $p_{i-1}$, $p_{i+1}$, and $x_i$. By assumption, $x_i\notin B$, leaving only $p_{i-1}$ and $p_{i+1}$. If $x_{i-1},x_{i+1} \in B$, then $p_{i-1},p_{i+1}\notin B$, so defining $c_j' = c_j$ for $j \neq i$ and $c_i' = 1$ yields a $\mathbf{c}'$ satisfying all polynomials of $B$. Suppose now that not both of $x_{i-1}$ and $x_{i+1}$ are in $B$. Condition (1) gives that at least one of $x_{i-1}$ and $x_{i+1}$ are in $B$, so $B$ cannot contain both of $p_{i-1}$ and $p_{i+1}$. Without loss of generality, suppose $p_{i-1}\in B$, giving that $x_{i+1}\notin B$. In this case, define $\mathbf{c}'$ so that $c'_j = c_j$ for $j \not \in \{i,i+1,i+2\}$, $c_i' = 1$, and $c_{i+1}' = -c_{i-3}c_{i-2}$. Then $p_{i-1}(\mathbf{c}') = 0$, so the only other polynomial containing $x_{i+1}$ is $p_{i+3}$. If $c_{i+1}' = 0$, then we already have $\mathbf{c}'\in V$. Suppose now that $c_{i+1}' \neq 0$. If $p_{i+3}\notin B$, then take $c_{i+2}' = c_{i+2}$, and $\mathbf{c}'\in V$. Otherwise, take $c_{i+2}' = -c_{i+4}c_{i+5}/c_{i+1}'$. In this case, $x_{i+3},p_{i+3}\in B$ gives that $x_{i+2}\notin B$. Furthermore,  $x_{i+3}\in B$ also implies $p_{i+1}\notin B$, meaning $p_{i+3}$ is the only polynomial of $B$ containing $x_{i+2}$. Therefore, in this case, $\mathbf{c}'\in V$.
\end{claimproof}

We will often use the above claim in contrapositive form, i.e., if $x_i\in B'$, then $x_i\in B$.\\

Suppose the polynomial $p_a$ is an element of $B' \setminus B$. Thus, $x_{a-2}\in B\setminus B'$ or $x_{a-1}\in B\setminus B'$ by condition (4). The same conclusion can be drawn of $x_{a+1}$ or $x_{a+2}$. Without loss of generality, there are three cases: $x_{a-2},x_{a+2} \in B \setminus B'$ and $x_{a-1},x_{a+1} \not \in B \setminus B'$, $x_{a-1},x_{a+1} \in B \setminus B'$, and $x_{a-1},x_{a+2} \in B \setminus B'$. Suppose, by way of contradiction, that $\Theta_V$ does not contain a set of odd cardinality greater than $1$.\\

Case $1$: $\{x_{a-2},x_{a+2}\}\subset B \setminus B'$ and $x_{a-1},x_{a+1} \not \in B \setminus B'$. Since $B'$ is an admissible set, then $x_a \in B'$, further implying $x_a \in B$ by the above claim. Let $M_a$ be the element of $\Theta_V$ containing $x_a$. By assumption, $|M_a|$ is even. If $M_a^L$ denotes the subset of variables in $M_a$ with indices less than $a$ and $M_a^R$ denotes the subset of variables in $M_a$ with indices greater than $a$, then exactly one of $|M_a^L|$ and $|M_a^R|$ is odd. Without loss of generality, suppose $|M_a^R|$ is odd, and let $x_{a+2q}$ be the variable of largest index in $M_a$. Clearly $q \geq 1$. 

Since $x_{a+2m}\in B$ for all $0\leq m\leq q$, we have that $x_{a+2m+1}\notin B$ for each $0\leq m\leq q-1$ by condition (4). The above claim gives that $x_{a+2m+1}\notin B'$ for each $0\leq m\leq q-1$. Therefore, $x_{a+2}, x_{a+3} \not \in B'$, so $p_{a+4} \in B'$ by condition (4), so $x_{a+6} \not \in B'$.  Repeating this argument, $B'\setminus B$ contains polynomials $p_i$ for $i\in\{a,a+4,a+8,\dotsc,a+2(q-1)\}$, since $|M_a^R|$ odd implies $q$ odd. Furthermore, $x_{a-2}, x_{a+2}, x_{a+6}, \dotsc, x_{a+2q} \notin B'$. If $a+2(q+1)\leq 2n-1$, then $x_{a+2q}$ being the variable with maximum index in $M_a$ implies $x_{a+2(q+1)}\notin B$. The above claim gives $x_{a+2(q+1)} \notin B'$, and this together with $x_{a+2q}\notin B'$ contradicts condition (1). Therefore, $a+2q = 2n+1$. Since $x_{a+2(q-1)},x_{a+2q}\in B$, then exactly one of $x_{2n}$ and $x_{2n+1}$ are not in $B$. Without loss of generality, suppose $x_{2n}\notin B$. By the above claim, we have that $x_{2n}\notin B'$. This together with $x_{a+2q}\notin B'$ contradicts condition (3), completing the case. \\

Case $2$: $\{x_{a-1},x_{a+1}\}\subseteq B \setminus B'$. By the definition of an admissible set, we have $x_a\notin B$ (as otherwise implies $x_{a-2}\notin B$ and $x_{a+2}\notin B$, giving that $p_a\in B$, a contradiction). The absence of $x_a$ in $B$ further implies that $x_a\notin B'$ by the above claim. If $3 < a < 2n - 1$, then $\{x_{a-2},x_{a+2}\}\subseteq B$. If $a = 3$ or $a = 2n-1$, suppose without loss of generality that $a = 3$, in which case $x_{a+2} \in B$. For any $a$, there exists $x_j$ with $j\in\{a+2,a-2\}$ so that $3\leq j\leq 2n-1$ and $x_j\in B$. The presence of $p_a\in B'$ requires $x_j\notin B'$. This together with $x_a\notin B'$ contradicts the definition of an admissible set. \\

Case $3$: Without loss of generality, $\{x_{a-1},x_{a+2}\}\subseteq B \setminus B'$. Suppose that $3 < a < 2n-1$. Since $B'$ is an admissible set, $x_{a+2}\notin B'$ implies $x_a\in B'$, so $x_a \in B$ by the above claim. Furthermore, $\{x_a,x_{a-1}\}\subset B$ implies $x_{a-2}\notin B$, giving that $x_{a-2}\notin B'$, again by the above claim. Let $M_a$ be the element of $\Theta_V$ containing $x_a$. We have that $x_a$ is the variable with smallest index in $M_a$, since $x_{a-2}\notin B$. Let $x_{a+2q}$ be the variable with largest index in $M_a$. Since $|M_a|$ is even, we have that $q\geq 1$ is odd. Therefore, applying the argument from case $1$ completes this case as well. \\

Since this considers all cases, this completes the proof that, if $V$ is not maximal, then $B$ contains a maximal odd order collection of monomials with consecutive indices in $A_n'$. 
\end{proof}

Let $\cJ_n$ denote the collection of all ideals in $\mathcal{I}_n$ which generate inclusion-maximal irreducible varieties. Furthermore, define $\cT_n$ to be the subcollection of $\bigcup_{S\in\mathcal{F}_n} \cT_S$ containing all admissible sets which generate ideals in $\cJ_n$. Lastly, define $\hat{\mathcal{F}}_n$ to be the subcollection of $\mathcal{F}_n$ containing all Fibonacci subsets of $A_n'$ which give rise to at least one admissible set in $\cT_n$, i.e., subsets $S$ of $A_n' = \{3,5,\ldots,2n-1\}$ so that at least one of every two consecutive elements of $A_n'$ belong to $S$, and so that maximal intervals of $A_n'$ contained in $S$ are either a single element or have even length. 

\begin{theorem}\label{Thm:IrredComps}
If $\cH = P_n^3$ for some $n\geq 3$, then the null variety $V_0$ of $\cH$ can be written $\cup_{J\in\cJ_n} \cV(J)$, where $\cJ_n$ is as defined above and each $J \in \cJ_n$ is an irreducible component of $V_0$.
\end{theorem}
\begin{proof}
Recall that the hyperpath $\cH$ has exactly $2n+1$ vertices, and we label them with $\{v_1,\ldots,v_{2n+1}\}$ so that the $j$-th edge is $e_j = \{v_{2(j-1)+1},\ldots,v_{2j+1}\}$ for $j = 1,\ldots,n$. 

In constructing the equations that define $V_0$, there are $n-1$ vertices giving rise to equations of the form $p_k = 0$, while the other $n+2$ vertices give equations of the form $x_ix_j = 0$. We begin by considering the variety defined by all polynomials of the second form. Let $x_{i_k}x_{j_k}$ for $1\leq k\leq n+2$ be the $n+2$ polynomials of this form. Then 
$$
\cV \left( \{x_{i_k}x_{j_k}\}_{k=1}^{n+2} \right) = \bigcap_{k=1}^{n+2} \cV(x_{i_k}x_{j_k}) = \bigcap_{k=1}^{n+2} \left( \cV(x_{i_k}) \cup \cV(x_{j_k}) \right) .
$$
Let $\ell_k \in \{i_k,j_k\}$ for each $1\leq k\leq n+2$, so that 
$$
\bigcap_{k=1}^{n+2} \left( \cV(x_{i_k}) \cup \cV(x_{j_k}) \right) = \bigcup_{\{\ell_k\}_{k=1}^{n+2}} \cV(\{x_{\ell_k}\}_{k=1}^{n+2}).
$$
Let $L$ be the collection of all choices of $\{\ell_k\}$. To facilitate analysis of the sets in $L$, we construct a graph $G$, where the vertices of $G$ are labeled with the distinct $\ell_k$, and edges connect $\ell_k$ and $\ell_{k'}$ if and only if $x_{\ell_k}x_{\ell_{k'}} \in \{x_{i_k}x_{j_k}\}_{k=1}^{n+2}$. Based on the structure of $\cH$ and the vertex labeling given originally, $G$ has the following form. 
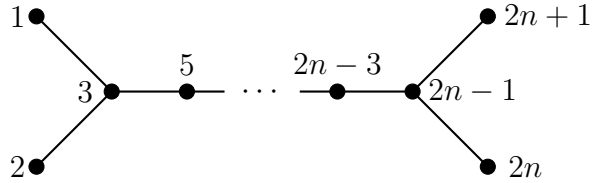
\begin{figure}[H]
\centering
\begin{tikzpicture}
\filldraw[black](-3,1)circle(0.1);
\filldraw[black](-3,-1)circle(0.1);
\filldraw[black](-2,0)circle(0.1);
\filldraw[black](-1,0)circle(0.1);

\filldraw[black](3,1)circle(0.1);
\filldraw[black](3,-1)circle(0.1);
\filldraw[black](2,0)circle(0.1);
\filldraw[black](1,0)circle(0.1);

\draw[thick,black] (-3,1) -- (-2,0) -- (-0.5,0);
\draw[thick,black] (-3,-1) -- (-2,0);
\draw[thick,black] (3,1) -- (2,0) -- (0.5,0);
\draw[thick,black] (3,-1) -- (2,0);
\node at (0,0) {$\cdots$};

\node at (-3.25,1) {$1$};
\node at (-3.25,-1) {$2$};
\node at (-2.35,0) {$3$};
\node at (-1,0.35) {$5$};

\node at (3.8,1) {$2n+1$};
\node at (3.5,-1) {$2n$};
\node at (2.8,0) {$2n-1$};
\node at (1,0.35) {$2n-3$};
\end{tikzpicture}
\end{figure}
An element of $L$ corresponds to a set of vertices in $G$ covering $E(G)$, since the vertices of $G$ are labeled by variable indices, edges are given by pairs of indices in a term of $\bigcap_{k=1}^{2n+1} \left( \cV(x_{i_k}) \cup \cV(x_{j_k}) \right)$, and $\bigcap_{k=1}^{2n+1} \left( \cV(x_{i_k}) \cup \cV(x_{j_k}) \right)$ is the union of intersections over one term from each element of $L$.

A subset $S$ of vertices in $G$ which is an edge cover must, in particular, cover the edges $\{3,5\}, \{5,7\}, \ldots, \{2n-3,2n-1\}$, so no two consecutive elements of $A_n'$ are absent from any such set.  In particular, $S \cap A_n' \in \cF_n$.  Let $\cX_S = \{x_i : i \in S\}$.  Since $3 \not \in S$ implies $1,2 \in S$ so that $S$ covers the edges $\{1,3\}$ and $\{2,3\}$, if $x_3 \not \in \cX_S$, then $x_1,x_2 \in \cX_S$.  Similarly, if $x_{2n-1} \not \in \cX_S$, then $x_{2n},x_{2n+1} \in \cX_S$.  Note that, for any odd $a$, if
\begin{equation} \label{eq1}
[(x_{a-2} = 0) \vee (x_{a-1}=0)] \wedge [(x_{a+1} = 0) \vee (x_{a+2}=0)]
\end{equation}
then $p_a = 0$.  Then let $P$ be the set of $p_a$ so that (\ref{eq1}) is {\em not} satisfied, and define $B = \cX_S \cup P$.  Then, for each $i \in A_n'$:
\begin{enumerate}
    \item If $i \not \in S$ and $i-4 \not \in S$, then $x_i \not \in B$, $x_{i-4} \not \in B$, and $p_{i-2} \in B$.
    \item If $i \not \in S$ and $i+4 \not \in S$, then $x_i \not \in B$, $x_{i+4} \not \in B$, and $p_{i+2} \in B$.
    \item If $i \not \in S$ and $i-4 \in S$, then $x_i \not \in B$, $x_{i-4} \in B$, and $x_{i-1} \in B$.
    \item If $i \not \in S$ and $i+4 \in S$, then $x_i \not \in B$, $x_{i+4} \in B$, and $x_{i+3} \in B$.
    \item If $5 \not \in S$ and $x_1 \not \in B$, then $p_3 \in B$.
    \item If $2n-3 \not \in S$ and $x_{2n+1} \not \in B$, then $p_{2n-1} \in B$.
    \item If $5 \in S$ and $x_1\notin B$, then $x_2\in B$. 
    \item If $5 \in S$ and $x_2\notin B$, then $x_1\in B$. 
    \item If $2n-3 \in S$ and $x_{2n}\notin B$, then $x_{2n+1}\in B$. 
    \item If $2n-3 \in S$ and $x_{2n+1}\notin B$, then $x_{2n}\in B$. 
\end{enumerate}

Let $\cB$ be the set of all such $B$ generated by the above conditions. Then, we have that the null variety of $\cH$ is $\cup_{B \in \mathcal{B}} \mathcal{V}(B)$, and it is easy to see that this is exactly the same as the construction given by $\bigcup_{I \in \cI_n} \cV(I)$.  Since Lemma \ref{Lem:IisPrime} gives that each of these ideals are prime, the corresponding varieties are irreducible, giving that $\bigcup_{I \in \cI_n} \cV(I)$ is a decomposition of $V_0$ into irreducible varieties. Furthermore, Lemma \ref{Lem:MaxComps} determines the inclusion-maximal varieties under the inclusion relation, implying that $\cup_{J\in\cJ_n} \cV(J)$ is a decomposition of $V_0$ into its irreducible components.
\end{proof}

\begin{cor}\label{Cor:GeoMult}
For $n\geq 3$, the null variety $V_0$ of $P_n^3$ has dimension $2 \lfloor n/2 \rfloor + 1$.
\end{cor}

As an illustration of Theorem \ref{Thm:IrredComps}, we list all the ideals that generate irreducible components of $V_0$ for $P_5^3$:
$$
\begin{array}{ccc}
\langle x_1,x_2,x_5,x_9, p_5, p_{9} \rangle & \langle x_1,x_2,x_4,x_5,x_7,x_8,x_{10},x_{11} \rangle & 
\langle x_3,x_7,x_{10},x_{11},p_3,p_7 \rangle \\
\langle x_3,x_6,x_7,x_{9},x_{10},p_3 \rangle &
\langle x_3,x_6,x_7,x_{9},x_{11},p_3 \rangle &
\langle x_1,x_3,x_5,x_6,x_{9},p_9 \rangle \\
\langle x_2,x_3,x_5,x_6,x_{9},p_9 \rangle &
\langle x_1,x_3,x_5,x_7,x_9,x_{10} \rangle &
\langle x_1,x_3,x_5,x_7,x_9,x_{11} \rangle \\
\langle x_2,x_3,x_5,x_7,x_9,x_{10} \rangle &
\langle x_2,x_3,x_5,x_7,x_9,x_{11} \rangle & 
\end{array}
$$

\subsection{Enumeration of Components by Dimension}

From here we work to determine the quantity of irreducible components of $V_0$ of different dimensions for each $P_n^3$. Fix an $n$. Let $B\in\mathcal{T}_n$, and let $S$ be such that $S \in\hat{\mathcal{F}}_n$ with $B$ an $S$-admissible set. Let $U_1, U_2,U_3, U_4$ be given so that $B = U_1 \cup U_2\cup U_3 \cup U_4$ as in the definition above. Noting that $|U_1| = |S|$, $|U_2| = \begin{cases}
2 & \mbox{if } x_3 \notin U_1 \\
1 & \mbox{otherwise}
\end{cases}$, $|U_3| = \begin{cases}
2 & \mbox{if } x_{2n-1} \notin U_1 \\
1 & \mbox{otherwise}
\end{cases}$, and $|U_4| = |\{a\in A_n: a-2\notin S\mbox{ or } a+2\notin S\}|$, the following computation gives an expression for $|B|$. 
\begin{align*}
|B| &= |S| + \begin{cases}
2 & \mbox{if } x_3 \notin U_1 \\
1 & \mbox{otherwise}
\end{cases} + \begin{cases}
2 & \mbox{if } x_{2n-1} \notin U_1 \\
1 & \mbox{otherwise}
\end{cases} + |\{a\in A_n: a-2\notin S\mbox{ or } a+2\notin S\}| \\
&= |S| + 1_{3\notin S} + 1_{2n-1\notin S} + |(A_n' - 2)\cap \overline{S}| + |(A_n' + 2)\cap \overline{S}| - |(A_n' - 2)\cap (A_n' + 2)\cap \overline{S}|\\
&= |S| + 1_{3\notin S} + 1_{2n-1\notin S} + |(A_n' - 2)\cap \overline{S}| + |(A_n' + 2)\cap \overline{S}| - |A_n\cap \overline{S}|\\
&= |S| + 1_{3\notin S} + 1_{2n-1\notin S} + |A_n'| - |(S-2)\cap (S+2)| \\
&= |S\cap A| + n+1 - |S\cap (S+4)| 
\end{align*}

Additionally, let $\mu_n(S)$ denote $|\mathcal{T}_{S} \cap \cJ_n|$, i.e., the number of irreducible components of $V_0$ generated by sets in $\mathcal{T}_S$. It is clear that $\mu_n(S) \in\{1,2,4\}$. All irreducible components generated by sets in $\mathcal{T}_{S}$ have dimension $2n+1 - |B|$ for some $B \in\mathcal{T}_{S}$, since the irreducible components all reside in $\mathbb{C}[x_1,\dotsc,x_{2n+1}]$, $|B_1| = |B_2|$ for all $B_1,B_2\in\mathcal{T}_{S}$, and the sets $B \in \cT_S$ are irredundant by Lemma \ref{Lem:IisPrime}. Consider the generating function 
\begin{align*}
g(y,z) &= \sum_{n\geq 0}\sum_{S \in\hat{\mathcal{F}}_n} y^{|B|} z^n.
\end{align*}
Note that $g(y,z)$ does {\em not} incorporate the multiplicity $\mu_n(S)$.  We first consider the expression given by the inner sum, namely 
$$
g_n(y) := \sum_{S \in \hat{\mathcal{F}}_n} y^{|B|}
$$
for a given $n \in \mathbb{N}$. Computation gives the following results for small values of $n$.
\begin{center}
\begin{minipage}{0.4\textwidth}
\begin{align*}
g_0(y) &= y \\
g_1(y) &= y^{2} \\
g_2(y) &= 2y^3 \\ 
\end{align*}
\end{minipage}
\begin{minipage}{0.4\textwidth}
\begin{align*}
g_3(y) &= 3y^4 \\ 
g_4(y) &= 3y^6 + y^4 \\ 
g_5(y) &= y^8 + 5y^6 \\ 
\end{align*}
\end{minipage}
\end{center}

We develop a recurrence for $g_n(y)$ aided by two new sequences of functions, $b_n(y)$ and $c_n(y)$, defined in the following way:
\begin{align*}
b_n(y) &= \sum_{S \in \hat{\mathcal{F}}_n,\{2n-3,2n-1\}\subseteq S} y^{|B|} \\
c_n(y) &= \sum_{S \in \hat{\mathcal{F}}_n,2n-3\notin S,2n-1\in S} y^{|B|} 
\end{align*}
For clarity, we define $b_0 = b_1 = b_2 = c_0 = c_1 = c_2 = 0$. Otherwise, we have the following small values of the two new sequences. 

\begin{center}
\begin{minipage}{0.4\textwidth}
\begin{align*}
b_3(y) &= y^4 \\
b_4(y) &= y^6 \\
b_5(y) &= 2y^6 \\ 
\end{align*}
\end{minipage}
\begin{minipage}{0.4\textwidth}
\begin{align*}
c_3(y) &= y^4 \\ 
c_4(y) &= y^4 \\ 
c_5(y) &= 2y^6 \\ 
\end{align*}
\end{minipage}
\end{center}

Note that, for each $S$ a Fibonacci subset of $A_n'$, at least one of $2n-3$ and $2n-1$ are included in $S$, so there are three options for $\{2n-1,2n-3\} \cap S$. All three can be expressed in terms of $b_n$, $c_n$, and $g_n$. A straightforward (if laborious) case analysis provides the following recurrences for the three sequences of functions. Note that these recurrences are valid only for $n\geq 5$.
\begin{align}
g_n(y) &= 2y^2g_{n-2}(y) + y^4b_{n-3}(y) + y^2(y^2-1)c_{n-2}(y) \label{eq2} \\
b_n(y) &= y^2g_{n-2}(y) - y^2c_{n-2}(y) \nonumber \\
c_n(y) &= y^2b_{n-2}(y) + y^2c_{n-2}(y) \nonumber
\end{align}
Recall that $g(y,z)$ is the generating function for $g_n(y)$. Analogously, let $b(y,z) = \sum_{n\geq 0} b_n(y)z^n$ and $c(y,z) = \sum_{n\geq 0} c_n(y)z^n$. The following computations work towards closed forms for $b$, $c$, and $g$.
\begin{align*}
g &= \sum_{n = 0}^4 g_nz^n + 2y^2 \sum_{n\geq 5} g_{n-2} z^n + y^4 \sum_{n\geq 5} b_{n-3} z^n + y^2(y^2-1)\sum_{n\geq 5} c_{n-2} z^n \\
&= \sum_{n = 0}^4 g_nz^n + 2y^2z^2 (g - \sum_{n=0}^2 g_{n} z^n) + y^4z^3b + y^2(y^2-1)z^2c \\
g &= \frac{3y^6z^4 - 4y^5z^4 + y^4z^4 + y^4z^3 + y^2z + y + y^4z^3b + y^2(y^2-1)z^2c }{ 1 - 2y^2z^2 } \\
b &= \sum_{n = 0}^4 b_n z^n + y^2 \sum_{n\geq 5} g_{n-2}z^n - y^2 \sum_{n\geq 5} c_{n-2} z^n = \sum_{n = 0}^4 b_n z^n + y^2z^2 (g - \sum_{n=0}^2 g_{n} z^n) - y^2z^2c \\
&= y^6z^4 - 2y^5z^4 - y^3z^2 + y^2z^2g - y^2z^2c \\
c &= \sum_{n=0}^4 c_n z^n + y^2 \sum_{n\geq 5} b_{n-2} z^n + y^2 \sum_{n\geq 5} c_{n-2} z^n = \sum_{n=0}^4 c_n z^n + y^2z^2b + y^2z^2c \\
c &= \frac{y^4z^4 + y^4z^3 + y^2z^2b}{1 - y^2z^2}
\end{align*}
Solving the system for $g$ gives the following. 
$$
g = -\frac{{\left(y^{8} - 2 \, y^{7} + y^{6}\right)} z^{6} - {\left(y^{8} - 2 \, y^{7}\right)} z^{5} - {\left(2 \, y^{6} - 3 \, y^{5} + y^{4}\right)} z^{4} + {\left(y^{5} - y^{4}\right)} z^{3} - y^{2} z - y}{y^{4} z^{4} - y^{4} z^{3} - 2 \, y^{2} z^{2} + 1}
$$

Recall that the exponent on $y$ in $g(y,z)$ is the co-dimension of the irreducible component of $V_0$ for $P_n^3$. Since we are interested in the dimension of these components, we make the following transformation. The dimension of each component is $2n+1$ minus its co-dimension. Thus, the function we want is given by $h(y,z) = y \cdot g(1/y,y^2z)$, expressible as follows (computations throughout performed by SageMath \cite{sagemath}).
\begin{align*}
h = \frac{-y^{7} z^{6} + 2 y^{6} z^{6} - y^{5} z^{6} + y^{5} z^{4} - 2 y^{4} z^{5} - 3 y^{4} z^{4} + y^{3} z^{5} + 2 y^{3} z^{4} + y^{3} z^{3} - y^{2} z^{3} + y z + 1}{y^{4} z^{4} - y^{2} z^{3} - 2 y^{2} z^{2} + 1}
\end{align*}
To help later with verifying Conjecture \ref{conj:HuYe}, differentiating with respect to $y$ gives the following expression and then plugging in $y=2$, because
$$
H(z) := \left . \frac{\partial}{\partial y} h(y,z) \right |_{y = 2} = \sum_{n\geq 0}\sum_{S \in\hat{\mathcal{F}}_n} (\dim \mathcal{V}(B)) 2^{\dim \mathcal{V}(B) -1 } z^n.
$$

The generating function obtained in this way encodes a lower bound on $\gm(0)$ of the conjecture, but four times this function is an upper bound. We get the following expression when substituting $y = 2$:
\begin{align*}
& \frac{-1280 z^{10} + 384 z^{9} + 1136 z^{8} + 192 z^{7} - 224 z^{6} - 132 z^{5} - 20 z^{4} + 20 z^{3} + 8 z^{2} + z}{256 z^{8} - 128 z^{7} - 240 z^{6} + 64 z^{5} + 96 z^{4} - 8 z^{3} - 16 z^{2} + 1} \\[.1in]
& \hspace{1in} = z + 8z^2 + 36z^3 + 116z^4 + 412z^5 + 1088z^6 + \cdots
\end{align*}

The smallest-magnitude root of the denominator lies in the interval $(0.37, 0.371)$. This implies that the coefficients of $H(z)$ have growth rate in the interval $(2.69, 2.71)$. We upper-bound the coefficients $\{\eta_n\}_{n \geq 0}$ of $H(z)$. Recall that Corollary \ref{Cor:GeoMult} gives that the maximum dimension of an irreducible component of $V_0$ for $P_n^3$ is $2 \lfloor n/2 \rfloor + 1$. Since we counted at most one component for each Fibonacci subset of $A_n'$, there are at most $F_n$ (the $n$-th Fibonacci number) terms which contribute to $\eta_n$. Therefore, $\eta_n$ is bounded above in the following way, given that $\phi = (1+\sqrt{5})/2$:
$$
\eta_n \leq \frac{\phi^n - (-\phi)^{-n}}{\sqrt{5}} \cdot (n+1)\cdot 2^n
$$

\subsection{Incorporating Multiplicity}

Recall that $\mu_n(S) \in \{1,2,4\}$ for $S \in \hat{\cF}_n$, but the above sums ignore this factor. Note that $\mu_n(S) > 1$ when either pair $\{3,5\}$ or $\{2n-3,2n-1\}$ are subsets of $S$. The sequence $b_n$ given above accounts for the subcollection of $\hat{\mathcal{F}}_n$ containing both $2n-3$ and $2n-1$, so $b$ is the generating function where the $y^m z^n$ coefficient counts the number of irreducible components of codimension $m$ from a hyperpath of length $n$ generated from a given $S$ containing both $2n-3$ and $2n-1$. By the symmetry of these Fibonacci subsets, the coefficients of $b$ also count the same quantity, where now the Fibonacci set $S$ contains both $3$ and $5$. So, $2b$ counts the $\{3,5\} \subseteq S$ and $\{2n-3,2n-1\}\not\subseteq S$ components once, the $\{2n-3,2n-1\}\subseteq S$ and $\{3,5\} \not\subseteq S$ components once, and the $\{3,5,2n-3,2n-1\}\subseteq S$ components twice. It only remains to count the $\{3,5,2n-3,2n-1\}\subseteq S$ components one additional time.

We now define $g_n'$, $b_n'$, and $c_n'$ to have the same conditions on the presence of $2n-3$ and $2n-1$ in $S$ as was given for $g_n$, $b_n$, and $c_n$ above, but now we require that $3$ and $5$ be in $S$, i.e.,
$$
g'_n(y) := \sum_{\substack{S \in \hat{\mathcal{F}}_n \\ 3,5 \in S}} y^{|B|}
$$
and analogously for $b'_n$ and $c'_n$.  We define all three sequences for $n\geq 0$, although some initial values are zero.  These modified sequences satisfy the exact same recurrences as displayed in (\ref{eq2}) for $n\geq 5$. 

Let $g'$, $b'$, and $c'$ be the generating functions with respect to the variable $z$ for the three sequences defined. Then, the generating function $b'$ counts exactly the $\{3,5,2n-3,2n-1\}\subseteq S$ components once. Computation gives the following rational expression for $g'$ and $b'$:
\begin{align*}
g' &= \frac{y^{6} z^{4} + y^{4} z^{3}}{y^{4} z^{4} - y^{4} z^{3} - 2 \, y^{2} z^{2} + 1} \\
b' &= -\frac{y^{6} z^{5} - y^{4} z^{3}}{y^{4} z^{4} - y^{4} z^{3} - 2 \, y^{2} z^{2} + 1} 
\end{align*}
Note that the generating function for $g'$ counts the same irreducible components as $b$ from above. Therefore, the generating function of $\gm(0)$, which incorporates multiplicity (aside from some initial terms), is given by $G = g + 2g' + b'$, and is given by the following rational function. 
\begin{align*}
G &= \left ( -y^{8} z^{6} + y^{8} z^{5} + 2 y^{7} z^{6} - 2 y^{7} z^{5} - y^{6} z^{6} - y^{6} z^{5} + 4 y^{6} z^{4} - 3 y^{5} z^{4} - y^{5} z^{3} + y^{4} z^{4} \right . \\
&\hspace{1in} \left . + \, 4 y^{4} z^{3} + y^{2} z + y \right )/(y^{4} z^{4} - y^{4} z^{3} - 2 y^{2} z^{2} + 1)
\end{align*}
Similarly to the previous subsection, we compute $h' = y \cdot G(1/y,y^2z)$, which is the generating function for the number of irreducible components of dimension given by the exponent on $y$ in $V_0$ for $P_n^3$, if $n$ is the exponent on $z$. 
\begin{align*}
h' &= \left ( -y^{7} z^{6} + 2 y^{6} z^{6} - y^{5} z^{6} - y^{5} z^{5} + y^{5} z^{4} - 2 y^{4} z^{5} - 3 y^{4} z^{4} + y^{3} z^{5} \right .\\
& \hspace{1in} \left . + \, 4 y^{3} z^{4} + 4 y^{3} z^{3} - y^{2} z^{3} + y z + 1 \right ) / (y^{4} z^{4} - y^{2} z^{3} - 2 y^{2} z^{2} + 1).
\end{align*}
Computing $\left . \frac{\partial}{\partial y} h'(y,z) \right |_{y=2}$ yields the following generating function:
\begin{align*}
&\frac{-1280 z^{10} + 128 z^{9} + 1200 z^{8} + 352 z^{7} - 336 z^{6} - 308 z^{5} + 4 z^{4} + 56 z^{3} + 8 z^{2} + z}{256 z^{8} - 128 z^{7} - 240 z^{6} + 64 z^{5} + 96 z^{4} - 8 z^{3} - 16 z^{2} + 1} \\
&= z + 8z^2 + 72z^3 + 140z^4 + 812z^5 + 1648z^6 + 7280z^7\\
& \qquad + 18064z^8 + 60928z^9 + 176576z^{10} + \cdots
\end{align*}
Here the linear and quadratic coefficients are incorrect, however, because incorporation of multiplicity only adjusts for $n \geq 3$. Modifying this expression via Propositions \ref{prop:oneedge} and \ref{prop:twoedge}, we obtain
\begin{align*}
H'(z) & = \frac{-256 z^{8} + 192 z^{7} + 272 z^{6} - 156 z^{5} - 92 z^{4} + 24 z^{3} + 13 z^{2} + 3 z}{256 z^{8} - 128 z^{7} - 240 z^{6} + 64 z^{5} + 96 z^{4} - 8 z^{3} - 16 z^{2} + 1} \\[.1in]
&= 3z + 13z^2 + 72z^3 + 140z^4 + 812z^5 + 1648z^6 + 7280z^7 \\
& \qquad + 18064z^8 + 60928z^9 + 176576z^{10} + \cdots
\end{align*}

\section{Algebraic Multiplicity of Zero}\label{Sec:AlgMult}

Let $D_{n,k}$ be the algebraic multiplicity of zero in the characteristic polynomial of $\phi_{P_n^k}(\lambda)$ (the $k$-uniform linear hyperpath with $n$ edges). We are given the following by the paper of Bao, Fan, Wang, and Zhu. 

\begin{theorem}[\cite{Bao20}] \label{Thm:FromBFWZ}
For $n\geq 2$, 
$$
\phi_{P_n^k}(\lambda) = \lambda^{(k-2)(k-1)^{n(k-1)}} \prod_{s=0}^n \left(\lambda - \frac{f^{s-1}(1)}{\lambda^{k-1}} \right)^{\nu_{n,k}(s)} \phi_{P_{n-1}^k}(\lambda)^{(k-1)^{k-1}},
$$
where 
$$
\nu_{n,k}(s) = 
\begin{cases}
k^{s(k-2)} ( (k-1)^{k-1} - k^{k-2} ) (k-1)^{(n-s-1)(k-1)} & \mbox{if } s\in[0,n-1], \\
k^{s(k-2)} & \mbox{if } s=n, \\
\end{cases}
$$
and 
$$
f^{i}(x) = 
\begin{cases}
0 & \mbox{if } i= -1, \\
1 & \mbox{if } i = 0, \\
f(x) = \frac{1}{1-x\lambda^{-k}} = \frac{\lambda^k}{\lambda^k - x} & \mbox{if } i = 1, \\
f^{i-1}(f(x)) & \mbox{if } i > 1.
\end{cases}
$$
\end{theorem}

We use these facts to prove the following. We start by proving the following lemma concerning the degree of the zero root in $f(x)$. 

\begin{lemma}\label{Lem:DegOff(1)}
Let $k\geq 2$ be given. Let $d_s$ be the degree of the zero root in the rational function $f^s(1)$. If $s\geq 1$, then $d_s = 0$ if $s$ is even and $d_s = k$ if $s$ is odd.  
\end{lemma}
\begin{proof}
We proceed by induction on $s$, with the base cases given by $s = 1$ and $s = 2$. The definition of $f^s(x)$ includes that $f(1) = \frac{\lambda^k}{\lambda^k-1}$, giving that $d_1 = k$. For $s = 2$, then,
$$
f^2(1) = f(f(1)) = f\left( \frac{\lambda^k}{\lambda^k-1} \right) = \frac{\lambda^k - 1}{\lambda^k - 2}.
$$
Now suppose that the result holds for some $s\geq 1$. Consider the value of $d_{s+1}$. Since composition of functions is associative, $f^{s+1}(1) = f^{s}(f(1)) = f(f^{s}(1))$. Let $q^s(x)$ denote the denominator of $f^s(x)$. Since $f(x) = \frac{\lambda^k}{\lambda^k-x}$, we can think of $f^{s+1}(1)$ as $\lambda^k q^s(1)$ divided by $\lambda^k q^s(1)$ minus the numerator of $f^s(1)$. 

If $f^{s}(1)$ is rational in $\lambda$ with $d_{s} = 0$, then the denominator of $f^{s+1}(1)$ will not be divisible by $\lambda$, but the degree of $\lambda$ in the numerator is $k$. Thus $d_{s+1} = k$.  On the other hand, if $d_{s} = k$, then $f^{s}(1)$ is rational in $\lambda$ with the power of $\lambda$ in the numerator equal to $k$. Then, $f^{s+1}(1)$ will have $k$ factors of $\lambda$ in the numerator after multiplying through by $q^s(1)$, but the denominator is the difference of two polynomials both of which have $\lambda$ occurring $k$ times as a factor. Factor out the term $\lambda^k$ from the denominator and cancel it within $f^{s+1}(1)$. This leaves zero factors of $\lambda$ in the numerator. In the denominator, we have zero factors of $\lambda$ if and only if the constant term in $q^s(1)$ differs from the coefficient of $\lambda^k$ in the numerator of $f^s(1)$. This inequality of coefficients is established by the following inductive argument, which need only handle the case of $s$ odd. In fact, we include in the inductive hypothesis as well that the numerator and denominator have no nonzero coefficients of terms of the form $\lambda^j$ with $0 < j < k$.

By definition, $f(1) = \frac{\lambda^k}{\lambda^k-1}$, so the constant term in the denominator (namely, $-1$) and the coefficient of $\lambda^k$ (namely, $1$) in the numerator differ, giving the base case. Suppose now that the result holds for some odd $i\geq 1$. Let $f^s(1)$ have numerator $\alpha(\lambda) + \alpha_1\lambda^k$ and denominator $\beta(\lambda) + \beta_1\lambda^k + \beta_2$, where $\alpha$ and $\beta$ are both polynomials of degree greater than $k$, and $\alpha_1 \neq \beta_2$. Then, we have the following. 
\begin{align*}
f^{s+2}(1) &= f\circ f\left( \frac{\alpha(\lambda) + \alpha_1 \lambda^k }{ \beta(\lambda) + \beta_1\lambda^k + \beta_2} \right) \\ 
&= f \left( \frac{\lambda^k}{\lambda^k - \frac{\alpha(\lambda) + \alpha_1\lambda^k}{\beta(\lambda) + \beta_1\lambda^k + \beta_2}} \right) \\ 
&= f \left( \frac{\beta(\lambda) + \beta_1\lambda^k + \beta_2}{\beta(\lambda) + \beta_1\lambda^k + \beta_2 - \alpha(\lambda)\lambda^{-k} - \alpha_1} \right) \\ 
&= \frac{\lambda^k}{\lambda^k - \left( \frac{\beta(\lambda) + \beta_1\lambda^k + \beta_2}{\beta(\lambda) + \beta_1\lambda^k + \beta_2 - \alpha(\lambda)\lambda^{-k} - \alpha_1} \right) } \\ 
&= \frac{\lambda^k (\beta(\lambda) + \beta_1\lambda^k + \beta_2 - \alpha(\lambda)\lambda^{-k} - \alpha_1 ) }{\lambda^k(\beta(\lambda) + \beta_1\lambda^k + \beta_2 - \alpha(\lambda)\lambda^{-k} + \alpha_1 ) - \beta(\lambda) - \beta_1\lambda^k - \beta_2} 
\end{align*}
From this, we see that the coefficient of $\lambda^k$ in the numerator is $\beta_2 - \alpha_1$, and the constant term in the denominator is $-\beta_2$. Since $\alpha_1 = 1$ and $\beta_2 = -1$ in $f(1)$, we have that the constant term in the denominator flips back and forth between $-1$ and $1$ as the powers of $f$ increase by two. On the other hand, $\beta_2 - \alpha_1$ takes values of the form $(-1)^{(s-1)/2}\cdot (s-1)/2$ for odd $s\geq 1$. Then the two desired coefficients are never equal, completing the proof.
\end{proof}

\begin{cor}\label{Cor:NuIsGood}
The multiplicity of the zero root of $\lambda^k - f^s(1)$ is the same as the multiplicity of zero in $f^s(1)$. 
\end{cor}
\begin{proof}
The even case is trivial, because both multiplicities are zero.  In the odd case, the ratio of the coefficient of $\lambda^k$ in the numerator of $f^s(1)$ divided by the constant coefficient in the denominator has absolute value less than $1$ except when $s=1$.  However, in that case $\lambda^k - f(1) = \lambda^k - \frac{\lambda^k}{\lambda^k-1} = \frac{\lambda^{2k} - 2\lambda^k}{\lambda^k - 1}$.
\end{proof}

We now use the preceding lemma and corollary to fully describe the nullity of $P_n^k$.

\begin{theorem}
Let $k\geq 1$ and $n\geq 1$. Additionally, let $u = (k-1)^{k-1}$ and $v = k^{k-2}$. If $D_{n,k}$ denotes the multiplicity of $\lambda$ in the $k$-uniform hyperpath characteristic polynomial $\phi_{P_n^k}(\lambda)$, then 
$$
D_{n,k} = \frac{u^n\left( [nk-n+1]u^2 + [nk-2n+2]uv - [k+n-1]v^2 \right) + k(-v)^{n+2}}{(u+v)^2}_.
$$
\end{theorem}
\begin{proof}
We first separate the $n = 1$ case. Cooper and Dutle \cite{Cooper12} showed that $D_{1,k} = k(k-1)^{k-1} - k^{k-1} = k(u-v)$. Plugging $n=1$ into the suggested formula gives the same expression, verifying the result for the base case. 
Suppose now that $n \geq 2$. From Theorem \ref{Thm:FromBFWZ}, we have 
\begin{equation} \label{eq3}
\phi_{P_n^k}(\lambda) = \lambda^{(k-2)(k-1)^{n(k-1)}} \prod_{s=0}^n \left(\lambda - \frac{f^{s-1}(1)}{\lambda^{k-1}} \right)^{\nu_{n,k}(s)} \phi_{P_{n-1}^k}(\lambda)^{(k-1)^{k-1}},
\end{equation}
so we develop a recurrence that gives $D_{n,k}$ knowing $D_{n-1,k}$. From the preceding formula, we see 
\begin{equation*}
D_{n,k} = (k-2)u^n + u\cdot D_{n-1,k} + F_{n,k},
\end{equation*}
where we define $F_{n,k}$ to be the multiplicity of the zero root in the simplified rational function $\prod_{s=0}^n \left(\lambda - \frac{f^{s-1}(1)}{\lambda^{k-1}} \right)^{\nu_{n,k}(s)}$ (taking the parameter to be negative if there are excess powers of $\lambda$ in the denominator). As above, let $d_{s}$ be the multiplicity of the zero root in $f^{s}(1)$. By Lemma \ref{Lem:DegOff(1)}, we have that $d_{s-1}$ is zero when $s-1$ is even, and $d_{s-1} = k$ when $s-1$ is odd. Since the $s$-th term of the product in (\ref{eq3}) is $[\lambda^{-(k-1)}(\lambda^k - f^{s-1}(1))]^{\nu_{n,k}(s)}$, and Corollary \ref{Cor:NuIsGood} gives that the degree of the zero root in $f^{s-1}(1)$ and $\lambda^k - f^{s-1}(1)$ are the same, we have 
$$
F_{n,k} = -(k-1)\sum_{s=0}^n \nu_{n,k}(s) + \sum_{s=0}^n \nu_{n,k}(s)\cdot d_{s-1}.
$$
We start by considering the value of the first term above. We have the following. 
\begin{align*}
\sum_{s=0}^n \nu_{n,k}(s) &= \nu_{n,k}(n) + \sum_{s=0}^{n-1} \nu_{n,k}(s) \\
&= v^n + \sum_{s=0}^{n-1} v^s (u-v)u^{n-s-1} \\
&= v^n + (u-v)u^{n-1} \frac{1 - \left( \frac{v}{u} \right)^n }{1 - \frac{v}{u}} \\
&= u^n
\end{align*}
When considering the second summand in the expression for $F_{n,k}$, we split into cases initially based on the parity of $n$. Starting with $n$ odd, we have the following simplification of $\sum_{s=0}^n \nu_{n,k}(s)\cdot d_{s-1}$: 
\begin{align*}
\sum_{s=0}^n \nu_{n,k}(s)\cdot d_{s-1} &= \sum_{s=0}^{(n-1)/2} \nu_{n,k}(2s) \cdot k \\
&= k\cdot \sum_{s=0}^{(n-1)/2} v^{2s} (u-v) u^{n-1-2s} \\
&= k(u-v)u^{n-1} \frac{1 - \left( \frac{v^2}{u^2} \right)^{(n+1)/2} } { 1 - \frac{v^2}{u^2} } \\
&=  \left( \frac{ k } { u + v } \right) (u^{n+1} - v^{n+1})
\end{align*}
On the other hand, if $n$ is even, we have the following. 
\begin{align*}
\sum_{s=0}^n \nu_{n,k}(s)\cdot d_{s-1} &= \sum_{s=0}^{n/2} \nu_{n,k}(2s) \cdot k \\
&= k\cdot v^n + k\cdot \sum_{s=0}^{(n-2)/2} v^{2s} (u-v) u^{n-1-2s} \\
&= k\cdot v^n + k(u-v)u^{n-1} \frac{1 - \left( \frac{v^2}{u^2} \right)^{(n)/2} } { 1 - \frac{v^2}{u^2} } \\
&=  \left( \frac{ k } { u + v } \right) (u^{n+1} + v^{n+1})
\end{align*}
Thus, for general $n$, we have 
$$
\sum_{s=0}^n \nu_{n,k}(s)\cdot d_{s-1} = \left( \frac{ k } { u + v } \right) (u^{n+1} - (-v)^{n+1}).
$$
This gives the following closed form for $F_{n,k}$. 
$$
F_{n,k} = -(k-1)u^n + \left( \frac{ k } { u + v } \right) (u^{n+1} - (-v)^{n+1})
$$
Substituting this back into the original expression for $D_{n,k}$, we have the following simplification. 
\begin{align*}
D_{n,k} &= (k-2)u^n + u\cdot D_{n-1,k} + F_{n,k} \\
&= (k-2)u^n + uD_{n-1,k} -(k-1)u^n + \left( \frac{ k } { u + v } \right) (u^{n+1} - (-v)^{n+1}) \\
&= uD_{n-1,k} + \frac{ u^n [(k-1) u - v] - k(-v)^{n+1} } { u + v }
\end{align*}
For $n = 1$, we noted earlier that $D_{1,k} = k(u-v)$. We continue with the following, completing the proof. 
\begin{align*}
D_{n,k} &= ku^{n-1}(u-v) + \frac{ u^n [(k-1) u - v] - k(-v)^{n+1} } { u + v } + \sum_{i=1}^{n-2} u^i \frac{ u^{n-i} [(k-1) u - v] - k(-v)^{n-i+1} } { u + v } \\
&= ku^{n-1}(u-v) + \sum_{i=0}^{n-2} u^i \frac{ u^{n-i} [(k-1) u - v] - k(-v)^{n-i+1} } { u + v } \\
&= ku^{n-1}(u-v) + \sum_{i=0}^{n-2} \frac{ u^{n} [(k-1) u - v] } { u + v } - \sum_{i=0}^{n-2} \frac{ ku^i(-v)^{n+1-i} } { u + v } \\
&= ku^{n-1}(u-v) + \frac{ u^{n} (n-1) [(k-1) u - v] } { u + v } - \frac{ k(-v)^{n+1} }{ u+v } \cdot \frac{ 1 - \left( \frac{ u } { -v } \right)^{n-1} } { 1 - \frac{u}{-v} } \\
&= \frac{ u^n ( [ nk - n + 1 ]u^2 + [nk - 2n + 2]uv - [k + n - 1] v^2 ) + k(-v)^{n+2} } { (u+v)^2 }.
\end{align*}
\end{proof}

The next result applies the above theorem to obtain an asymptotic expression for $D_{n,k}$.

\begin{cor}
Let $k\geq 3$ be fixed and $n\geq 1$. Then $\lim_{n\to\infty} \frac{D_{n,k}}{n(k-1)^{n(k-1)+1}} = 1$.  In particular, the fraction of eigenvalues of $P_n^k$ which are zero approaches $1/(k-1)$ as $n \rightarrow \infty$.
\end{cor}
\begin{proof}
We have the following expression for $D_{n,k}$, where $u = (k-1)^{k-1}$ and $v = k^{k-2}$:
$$
D_{n,k} = \frac{u^n\left( [nk-n+1]u^2 + [nk-2n+2]uv - [k+n-1]v^2 \right) + k(-v)^{n+2}}{(u+v)^2}
$$
Noting that $k \geq 2$, we first show that $u > v$. We have the following computation. 
$$
\frac{u}{v} = \frac{(k-1)^{k-1}}{k^{k-2}} = \frac{k^2}{k-1} \left( 1 - \frac{1}{k} \right)^k \geq \frac{k^2}{k-1} \cdot \frac14 = \frac{k^2}{4k-4}.
$$
Note that for $k \geq 2$, the function $\left( \frac{k-1}{k} \right)^k$ is increasing, so its value for any $k\geq 2$ is bounded below by its value when $k = 2$, namely, $1/4$. Furthermore, the rightmost expression is greater than one if and only if $k^2 \geq 4k - 4$, which is true because $(k-2)^2 \geq 0$. Therefore, $u > v$, so $u$ dominates $v$ asymptotically. Then the rational expression is asymptotically the same as a ratio of two polynomials just in the variable $u$, from which it follows that 
$$
\lim_{n\to\infty} \frac{D_{n,k}}{[n(k-1) + 1]u^n} = 1.
$$
Since $k$ is constant, this gives the desired result.  The second claim in the proof follows because (see \cite{Qi05}) the total number of eigenvalues (counted with algebraic multiplicity) is $N (k-1)^{N}$, where $N$ is the number of vertices; in this case, $N = n(k-1)+1$ and, as $n \rightarrow \infty$,
$$
\frac{n(k-1)^{n(k-1)+1}}{(n(k-1)+1) (k-1)^{n(k-1)+1}} \sim \frac{n(k-1)^{n(k-1)+1}}{n(k-1)^{n(k-1)+2}} = \frac{1}{k-1}.
$$
\end{proof}

From this, we observe the following lower bound for $D_{n,3}$ when $n\geq 12$. 
$$
D_{n,3} \geq \frac{4^n}{7} (5n+3)
$$

\section{Conjecture Verification}

\begin{theorem}
Let $V_0^1,\dotsc, V_0^\kappa$ denote the irreducible components of $V_0$ for $P_n^3$. For $n \geq 1$, $D_{n,3} \geq \sum_{i=1}^\kappa \dim(V_0^i)(2)^{\dim(V_0^i) - 1}$. 
\end{theorem}
\begin{proof}
Recall the following bounds on $D_{n,3}$ and $\eta_n$, where $\eta_n$ is the $z^n$ coefficient of the generating function $H(z)$ found in Section \ref{Sec:GeoMult}. 
\begin{align*}
D_{n,3} & \geq \frac{4^n}{7} (5n+3) \\
\eta_n &\leq \frac{\phi^n - (-\phi)^{-n}}{\sqrt{5}} \cdot (n+1)\cdot 2^n
\end{align*}
It is easy to check that $4(\phi^n + 1) \leq 2^n$ for any $n\geq 7$. Furthermore, 
$$
2^n \leq \frac{4\sqrt{5}}{7} \cdot 2^n \leq \frac{\sqrt{5}}{7} \cdot 2^n\cdot \frac{5n+3}{n+1}
$$
$$
4(\phi^n + 1) \geq 4(\phi^n - (-\phi)^{-n})
$$
Combining the inequalities shows that $D_{n,3} \geq 4\eta_n$ for $n\geq 12$:
\begin{align*}
4(\phi^n - (-\phi)^{-n}) &\leq \frac{\sqrt{5}}{7} \cdot 2^n\cdot \frac{5n+3}{n+1} \\
4\cdot \frac{\phi^n - (-\phi)^{-n}}{\sqrt{5}} \cdot (n+1) &\leq \frac{2^n}{7} \cdot (5n+3) \\
4\cdot \frac{\phi^n - (-\phi)^{-n}}{\sqrt{5}} \cdot (n+1)\cdot 2^n &\leq \frac{4^n}{7} \cdot (5n+3)
\end{align*}
Therefore, this gives us that the conjecture holds for $n\geq 12$, since $\gm(0) \leq 4\eta_n$. The following table computes values for $n < 12$ exactly, completing the proof. 
\begin{center}
\begin{tabular}{ c | c c c c c c c c c c c }
 n & $1$ & $2$ & $3$ & $4$ & $5$ & $6$ & $7$ & $8$ & $9$ & $10$ & $11$ \\
 \hline 
 $D_{n,3}$ & $3$ & $35$ & $151$ & $891$ & $3983$ & $19795$ & $88071$ & $407531$ & $1792063$ & $7993155$ & $34740791$ \\ 
 \hline 
 $\gm(0)$ & $3$ & $13$ & $72$ & $140$ & $812$ & $1648$ & $7280$ & $18064$ & $60928$ & $176576$ & $509376$
\end{tabular}
\end{center}
\end{proof}

\section{Conclusion}

We have shown how to compute the multiplicities of the zero eigenvalue of linear hyperpaths of rank $3$ and the dimensions of the irreducible components of the corresponding nullvarieties.  This enables us to compute $\gm(0)$ and $\am(0)$ so that they can be compared in order to verify the Hu-Ye conjecture in this special case.  The above analysis can be extended by straightforward generalization to higher rank hyperpaths.  Furthermore, some of the issues encountered in carrying out this analysis invite new questions:

\begin{enumerate}
    \item In general, linear hypertrees have many of the properties taken advantage of above for hyperpaths.  Therefore, we ask: can these methods be used to answer Conjecture \ref{conj:HuYe} for this much more general class of hypergraphs?
    \item The set of vertices/coordinates where hypergraphs' nullvectors are zero is combinatorially interesting and plays an integral role in our classification of irreducible components.  For example, it is not hard to see that these vertex sets are transversals of the hypergraph's edge set when it is any hypertree.  What is possible to say in general about these sets and their relation to the nullvariety's components?
    \item We have made no attempt to understand eigenvarieties corresponding to nonzero eigenvalues $\lambda$, nor have we attempted to compute their algebraic multiplicities.  Doing so would require answering: what are the rest of the root multiplicities of the characteristic polynomials of linear hyperpaths, and how does the structure of non-zero eigenvarieties differ from the nullvariety?
    \item Is Conjecture \ref{conj:HuYe} tight?  The quantity proposed for $\gm(\lambda)$ is perhaps not the maximum function of the multiset of eigenvariety component dimensions which still provides a lower bound on $\am(\lambda)$.
\end{enumerate}

\section{Acknowledgement}

The authors wish to thank Fan Chung for so much of the inspiration that paved the road to this subject, and to the organizers of the December 2019 TSIMF conference for creating the occasion to honor her and continue her work.

\bibliographystyle{plain}
\bibliography{refs}

\Addresses

\end{document}